\documentclass[11pt]{article}

\usepackage{amssymb, amsmath, amsthm, graphicx, xcolor}
\usepackage[left=1in,top=1in,right=1in]{geometry}
\date{}

\theoremstyle{plain}
      \newtheorem{theorem}{Theorem}[section]
      \newtheorem{lemma}[theorem]{Lemma}

      \newtheorem{conjecture}[theorem]{Conjecture}
\theoremstyle{definition}

\theoremstyle{remark}

 \newcommand{\twr}{{{\rm twr}}}

\title{Ramsey numbers of cliques versus monotone paths}
\author{Dhruv Mubayi\thanks{Department of Mathematics, Statistics, and Computer Science, University of Illinois, Chicago, IL, 60607 USA.  Research partially supported by NSF awards DMS-1763317, DMS-1952767, DMS-2153576, a Humboldt Research Award and a Simons Fellowship.   Email: {\tt mubayi@uic.edu}.} \and Andrew Suk\thanks{Department of Mathematics, University of California at San Diego, La Jolla, CA, 92093 USA. Supported by NSF CAREER award DMS-1800746 and  NSF award DMS-1952786. Email: {\tt asuk@ucsd.edu}.}}

\begin{document}

\maketitle

\begin{abstract} 
One formulation of the Erd\H os-Szekeres monotone subsequence theorem states that  for any red/blue coloring of the edge set of the complete graph on $\{1, 2, \ldots, N\}$, there exists  a monochromatic red $s$-clique or a monochromatic blue increasing path $P_n$ with $n$ vertices, provided $N >(s-1)(n-1)$. 
Here, we prove a similar statement as above in the off-diagonal case for triple systems, with the quasipolynomial bound $N>2^{c(\log n)^{s-1}}$. For the $t$th power $P_n^t$ of the ordered increasing graph path with $n$ vertices, we prove a near linear bound  $c\, n(\log n)^{s-2}$ which improves the previous  bound that applied to a more general class of graphs than $P_n^t$ due to Conlon-Fox-Lee-Sudakov.
\end{abstract}

\section{Introduction}

A well-known theorem of Erd\H os and Szekeres \cite{es} states that any sequence of $(n-1)^2 + 1$ distinct real numbers contains a monotone subsequence of length at least $n$.  This is a classical result in combinatorics and its generalizations and extensions have many important consequences in geometry, probability, and computer science. See Steele \cite{st} for 7 different proofs along with several applications.  Here, we study its extension in the ordered hypergraph setting.

 An \emph{ordered} $k$-uniform hypergraph $H$ on $n$ vertices is a hypergraph whose vertices are ordered  $\{1,2,\ldots,n\}$. Given two ordered $k$-uniform hypergraphs $G$ and $H$, the Ramsey numbers $r_k(G,H)$ is the minimum $N$ such that for every red/blue coloring of the $k$-tuples of $\{1,2,\ldots, N\}$, there is either a red copy of $G$ or a blue copy of $H$. When $G = H$, we simply write $r_k(H) = r_k(H,H)$.  We let $r_k(H;q)$ to be the minimum integer $N$ such that for every $q$-coloring of the $k$-tuples of $[N]= \{1,2,\ldots, N\}$, there is a monochromatic copy of $H$.  We write $K^{(k)}_n$ for the complete $k$-uniform hypergraph on $n$ vertices.   A \emph{monotone path of size $n$}, denoted by $P^{(k)}_n$, is an ordered $k$-uniform hypergraph whose vertex set is $\{1,2,\ldots, n\}$, and $n-k + 1$ edges of the form $(i,i + 1,\ldots, i + k - 1)$, for $i = 1,\ldots, n-k + 1$.  In order to avoid the excessive use of superscripts, we remove them when the uniformity is clear.   For example, we write $r_k(K_s,P_n)=  r_k(K^{(k)}_s,P^{(k)}_n)$.

The proof of the Erd\H os and Szekeres monotone subsequence theorem, and also  Dilworth’s theorem on partially ordered sets \cite{d}, implies that

$$r_2(K_s,P_n) = (s-1)(n-1) + 1. $$

\noindent  However for $k$-uniform hypgeraphs, when $k \geq 3$, $r_k(K_s,P_n)$ is much less understood.  In \cite{ms}, the authors showed a surprising connection between $r_{k}(K_s,P_{n})$ and the classical Ramsey number $r_{k - 1}(K_s;q)$.  More precisely, they showed that for $q \geq 2$

\begin{equation}\label{msbound}
    r_{k-1}(K_{\lfloor s/q\rfloor };q)  \leq r_{k}(K_{s}, P_{q + k - 1}) \leq r_{k -1}(K_s;q).
\end{equation}

\noindent   Hence, for $q = 2$, $k =O(1)$, and $s$ tending to infinity, determining the tower growth rate of $r_k(K_s,P_{k + 1})$ is equivalent to determining the tower growth rate of the classical Ramsey number $r_{k-1}(K_s)$.  Classical results of Erd\H os \cite{e} and Erd\H os and Szekeres \cite{es} imply that $r_2(K_s)   = 2^{\Theta(s)}$ (see also \cite{sp,sah}).  Unfortunately for $k$-uniform hypergraphs, when $k\geq 3$, there is an exponential gap between the best known lower and upper bounds for $r_k(K_s)$.  More precisely, 

$$\twr_{k - 1}(\Omega(s^2)) < r_k(K_s) < \twr_k(O(s)),$$

\noindent where the tower function $\twr_k(x)$ is defined recursivly by $\twr_1(x) = x$ and $\twr_{i+1}(x) = 2^{\twr_i(x)}$
(see \cite{EH,EHR,ER}).  A notoriously difficult conjecture of Erd\H os, Hajnal, and Rado states that the upper bound is the correct tower growth rate.
 
Unfortunately, (\ref{msbound}) doesn't shed much light on $r_k(K_s,P_n)$ when $s$ is fixed, and $n$ tends to infinity.  In this direction, the first author~\cite{m} showed that $r_3(K_4,P_n) = O(n^{21})$ and made the following conjecture.

\begin{conjecture}
We have $r_3(K_s,P_n) = O(n^c)$, where $c = c(s)$.
\end{conjecture}

\noindent Our first result establishes a quasi-polynomial bound for $r_3(K_s,P_n)$, when $s$ is fixed.  Throughout this paper, all logarithms are in base 2.

\begin{theorem} \label{quasipoly}
We have $r_3(K_s,P_n)< 2^{c_s(\log n)^{s-1}},$ where $c_s = 5^s s!.$
\end{theorem}

\noindent Together with the well-known neighborhood chasing argument of Erd\H os and Rado \cite{ER}, we have the following.

\begin{theorem}
  For $k \geq 3$, we have  $r_k(K_s,P_n) = \twr_{k-2}\left(2^{c(\log n)^{s-1}}\right),$ where $c = c(s)$.
\end{theorem}

In the other direction, we have the trivial inequality $r_k(K_s,P_n) \geq r_k(P_s,P_n)$. The famous cups-caps theorem of Erd\H os and Szekeres \cite{es} states that $r_3(P_s,P_n) = \binom{s + n -4}{s-2} + 1,$ and the stepping-up lemma established in \cite{fpss} (see Theorem 4.3) implies that $r_k(P_s,P_n) \geq \twr_{k-2}(n^c)$, where $c = c(s)$.  Thus, we essentially determine the tower growth rate of $r_k(K_s,P_n)$ for $s$ fixed and $n$ tending to infinity.

For the diagonal case $r_3(K_n, P_n)$, these observations and a result of the authors~\cite{ER} yield 
$$2^{n} < {2n-4 \choose n-2}= r_3(P_n, P_n) \le r_3(K_n, P_n) < r_2(n;n) < 2^{n^2 \log n}.$$
It would be interesting to  improve either  bound for $r_3(K_n, P_n)$.

\subsection{Cliques versus power paths in graphs} 

A key lemma in the first author's~\cite{m} proof of $r_3(K_4,P_n) = O(n^{21})$  is based on the following generalization of monotone paths in ordered graphs.  Given positive integers $t, n$, the \emph{$t$-th power of the path of $P_n$}, denoted by $P_n^t$, is an ordered graph with vertex set $\{1,2,\ldots, n\}$, and $(i,j)$ is an edge if and only if $|j - i| \leq t$.  Hence, $P_n^1 = P_n$.  In \cite{bckk}, Balko et al.~showed that $r_2(P_n^t) = O(n^{129t})$ (see also \cite{m}).  Our next result establishes a near linear bound in the off-diagonal setting.  Moreover, our proof generalizes to the clique versus power-path setting.

\begin{theorem}\label{main1}
For positive integers $s,t,n$ such that $t \leq s$, we have

$$r_2(P^t_s, P^t_n) \leq r_2(K_s,P_n^t) < t^{4s}n(\log n)^{s-2}.$$

\end{theorem}

For large $s$, e.g., $s = n$, we also have the following bound.

\begin{theorem}\label{main2}
For positive integers $s,t,n$, we have
$$r_2(K_s,P_n^t) < (2s)^{t(t + 1)\log n}.$$

\end{theorem}

\noindent Hence in the diagonal setting, for fixed $t > 0$, we have $r_2(K_n,P_n^t) \leq 2^{O(\log^2n)}.$  This coincides with a more general result established by Conlon, Fox, Lee, and Sudakov \cite{cfls} on ordered graphs with bounded degeneracy.   In the off-diagonal case, we make the following stronger conjecture.

\begin{conjecture}
    For all $s,t>1$ there exists $c=c_{s,t}$ such that
    $r_2(K_s, P^t_n) < c\, n$.
\end{conjecture}

\section{Non-increasing sets: Proof of Theorem \ref{quasipoly}}

In this section, we prove Theorem \ref{quasipoly} by establishing a Ramsey-type result for non-increasing sets.  Let $\chi$ be a $q$-coloring of the pairs of $[N]$, with colors $\{\kappa_1,\ldots, \kappa_q\} \subset \mathbb{Z}$ such that $\kappa_1 < \cdots < \kappa_q$.  Then we say that a triple $u, v,w \in [N]$, where $u < v < w$, is \emph{non-increasing} if

\begin{enumerate}
    \item $\chi(u,v)  = \chi(u,w) \geq \chi(v,w)$, or

    \item $\chi(u,v) \geq \chi(v,w) = \chi(u,w)$.
\end{enumerate}

We say that a set $S\subset [N]$ is \emph{non-increasing} with respect to $\chi$ if every triple in $S$ is non-increasing.  Given subsets $S,T \subset [N]$ such that $S = \{v_1,\ldots, v_s\}$ and $T = \{u_1,\ldots, u_s\}$, we say that $S$ and $T$ have the same \emph{color pattern} with respect to $\chi$ if $\chi(v_i,v_j) = \chi(u_i,u_j)$ for all $i,j$.

We will need the following lemma about non-increasing sets.



\begin{lemma}\label{order}
Let $S = \{v_1,\ldots, v_s\}$ be a non-increasing set with respect to $\chi$, where $v_1 < \cdots <  v_s$.  Fix vertex $v_j \in S$. 
  Then for any $v_i,v_{\ell} \in S$ such that $v_i < v_j < v_{\ell}$, we have

  \begin{enumerate}
      \item $\chi(v_{i},v_j) \geq \chi(v_j,v_{\ell})$, and

      \item  $\chi(v_{j-1},v_j) \leq \chi(v_i,v_j)$, and  

      \item $\chi(v_{j},v_{j + 1}) \geq \chi(v_j,v_{\ell})$.
  \end{enumerate}

\end{lemma}

\begin{proof}
The first property follows from the fact that $S$ is non-increasing.  For the second property, for sake of contradiction, suppose there is a vertex $v_i < v_{j-1}$ such that $\chi(v_{j-1},v_j) > \chi(v_i,v_j)$. Then we must have $\chi(v_i,v_{j-1}) = \chi(v_i,v_{j})$, contradicting the fact that $\{v_i,v_{j-1},v_j\}$ is non-increasing.  A similar argument shows that the third property follows.
\end{proof}

\noindent Let $f(s;q)$ be the minimum integer $N$, such that if the pairs of $[N]$ are colored with at most $q$ colors $\kappa_1 < \cdots < \kappa_q$, then there is a set $S\subset [N]$ of size $s$ such that every triple in $S$ is non-increasing.

\begin{theorem}\label{reduce}

We have $r_3(K_s,P_n) \leq f(s;n-2)$.

\end{theorem}

\begin{proof}

 Let $N = f(s;n)$ and let $\phi$ be a red-blue coloring of the triples of $[N]$.  If $\phi$ produces a blue monotone path of size $n$, then we are done.  Otherwise, we define $\chi:\binom{[N]}{2}\rightarrow \{2,3,\ldots, n-1\}$ such that for $u,v \in [N]$, $\chi(u,v)$ is the size of the longest blue monotone path ending at $(u,v)$ with respect to $\phi$.  By definition of $f(s;n)$, there is a set $S\subset [N]$ of $s$ vertices such that every triple in $S$ is non-increasing with respect to $\chi$.  Hence, $\phi$ must color every triple in $S$ red, which yields a red $K_s$ with respect $\phi$.\end{proof}

We now prove the following upper bound for $f(s;n)$.  Together with Theorem \ref{reduce}, Theorem \ref{quasipoly} quickly follows.

\begin{theorem}
    For $s\geq 3$ and $n\geq 2$, we have $f(s;n) \leq 2^{5^ss!(\log n)^{s-2}}$.
\end{theorem}

\begin{proof}
    We proceed by double induction on $s$ and $n$.  For the base case $n = 2$ and $s\geq 3$, we have

    $$f(s;2) \leq r_2(K_s) < 4^s < 2^{5^ss!}.$$

Therefore, let us assume that the statement holds for $n' < n$.  For the other base case $s = 3$ and $n \geq 2$, let $N = 2^{5^3\cdot 6\log n}$ and $\chi$ be an $n$-coloring of the pairs (edges) of $[N]$ with colors $\{1,\ldots, n\}$.  We can assume at least half of the edges have color $i \leq n/2$, since otherwise a symmetric argument would follow.  Let $E\subset \binom{[N]}{2}$ be the set of edges with color at most $n/2$, and for $v \in [N]$, let 

$$N^-_E(v) = \{u\in [N]: u < v,(u,v) \in E\},$$

\noindent and $d^-_E(v) = |N^-_E(v)|$. Hence, $\sum_v d^-_E(v) = |E|   \geq (1/2)\binom{N}{2}.$

By averaging, there is a vertex $v \in[N]$ such that $d_E^-(v) \geq (N-1)/4.$  If there is a pair in $N^-_E(v)$ with color $j > n/2$, then we have a non-increasing triple and we are done.  On the other hand, if no such pair has color $j > n/2$, since we have

$$|N^-_E(v)| \geq \frac{N-1}{4}  > 2^{5^3\cdot 6\log (n/2)},$$

\noindent we can apply induction to find a non-increasing triple and we are done.  

For the inductive step, let us assume that the statement holds for $s' < s$ and $n' < n$.  Let $N = 2^{5^ss!(\log n)^{s-2}}$.  Let $\chi$ be an $n$-coloring of the pairs of $[N]$ with colors $\{1,\ldots, n\}$.  By a standard supersaturation argument, we have at least

$$\frac{\binom{N}{f(s-1;n)}}{\binom{N - (s -1)}{f(s-1;n) - (s-1)}} \geq \frac{(N - s)^{s - 1}}{f(s-1;n)^{s-1}} \geq  \frac{N^{s - 1}}{2f(s-1;n)^{s-1}},$$

\noindent copies of a non-increasing set on $s-1$ vertices.  By the pigeonhole principle, there are at least $N^{s-1}/(2n^{s^2}f(s-1;n)^{s-1})$ non-increasing sets on $s-1$ vertices with the same color pattern.  Let us fix one such non-increasing set $S = \{v_1,\ldots, v_{s-1}\}$ for reference, and let $\chi(v_i,v_{i + 1}) = \kappa_i$.  
For convenience, set $\kappa_0 = n$ and $\kappa_{s-1} = 1$, which implies $$n = \kappa_0 \geq  \kappa_1  \geq \cdots \geq \kappa_{s-2} \geq \kappa_{s-1} = 1.$$    By the pigeonhole principle, there is an $i$ such that $1 \leq i \leq s-1$ such that $\kappa_{i-1} - \kappa_i \geq n/s$.  Since we have $N^{s-1}/(2n^{s^2}f(s-1;n)^{s-1})$ non-increasing sets on $s-1$ vertices with the same color pattern as $S$, there is a subset $B\subset [N]$ and $s-2$ vertices $u_1,\ldots,u_{i-1},u_{i + 1},\ldots, u_{s-1} \in [N]$ such that for each $b \in B$, we have

\begin{enumerate}
    \item $u_1 < \cdots < u_{i-1} < b < u_{i + 1} < \cdots < u_{s-1}$,

    \item $|B| \geq N/(2n^{s^2}f(s-1;n)^{s-1})$, and

    \item $S' = \{u_1,\ldots,u_{i - 1},b,u_{i + 1},\ldots,  u_{s-1}\}$ is non-increasing with the same color pattern as $S$.  
    
\end{enumerate} 

\noindent Let us remark that if $i = 1$, then we have $b < u_2 < \cdots < u_{s-1}$ for all $b \in B$, and $S' = \{b,u_2,\ldots, u_{s-1}\}$.  Likewise, if $i = s-1$, then we have $u_1 < \cdots < u_{s-2} < b$ for all $b \in B$, and $S' = \{u_1,\ldots, u_{s-2},b\}$.

If there is a pair $b,b' \in B$ such that $\kappa_{i-1} \geq \chi(b,b') \geq \kappa_i$, then the set $$T =  \{u_1,\ldots,u_{i - 1},b,b',u_{i + 1},\ldots,  u_{s-1}\}$$ is a nonincreasing set of size $s$. 
 Indeed, it suffices to check that triples of the form $\{u_j,b,b'\}$ for $j \leq i-1$, and $\{b,b',u_j\}$ where $j\geq i + 1$, are non-increasing.  Assume $j\leq i-1$.  By construction, we have $\chi(u_j,b) = \chi(u_j,b')$.   By Lemma \ref{order} and the assumption above, we have $$\chi(u_j,b) = \chi(u_j,b')\geq \kappa_{i-1}\geq\chi(b,b').$$  Hence, $\{u_j,b,b'\}$ is non-increasing.  For $j\geq i + 1$, a similar argument shows that $\{b,b',u_j\}$ is non-increasing.

Therefore, we can assume that $\chi$ uses at most $n - n/s = n(s-1)/s$ distinct colors on $B$.  However, this implies

\begin{eqnarray*}
|B|  & \geq &  \frac{N}{2n^{s^2}f(s-1;n)^{s-1}}\\\\
 & \geq & \frac{2^{5^ss!(\log n)^{s-2}}}{2n^{s^2}2^{(s-1)5^{s-1}(s-1)!(\log n)^{s-3}}} \\\\
&\geq &  2^{5^ss!(\log n)^{s-2} - 2(s-1)5^{s-1}(s-1)!(\log n)^{s-3}}\\\\
 & \geq & 2^{5^ss!(\log n - \log(s/(s-1)))^{s-2}} \\\\
& \geq & 2^{5^ss!(\log((s-1)n/s))^{s-2}}  \\\\
 & \geq & f(s;(s-1)n/s).
\end{eqnarray*}

\noindent By the induction hypothesis, we can find a non-increasing set inside of $B$.

\end{proof}

\section{Ordered graphs}

\begin{proof}[Proof of Theorem \ref{main1}]

We proceed by double induction on $s$ and $n$.  The base cases when $s = 2$ or when $n = 2$ is trivial.  For the inductive step, assume that the statement holds for $s' < s$ or $n' < n$.  Let $N = 8s^tn(\log n)^{s-2}$, and $V = [N]$. For sake of contradiction, suppose there is $\chi:\binom{[N]}{2}\rightarrow \{\textnormal{red,blue}\}$, such that $\chi$ does not produce a red $K_s$ nor a blue $P_n^t$.  Then we define

\begin{itemize}
    \item $U = \{\lfloor N/2\rfloor + 1, \lfloor N/2\rfloor + 2, \ldots, \lfloor N/2\rfloor  + \binom{s + t}{t}\}$, 
    
    \item $V_1 = \{1,2,\ldots, \lfloor N/2\rfloor\}$, 
    
    \item $V_2 = \{\lfloor N/2\rfloor  + \binom{s + t}{t} + 1, \lfloor N/2\rfloor  + \binom{s + t}{t} + 2,\ldots, N\}$
\end{itemize}

By Ramsey's theorem, we know that $r_2(K_s,K_t) < \binom{s + t}{t}$.  Hence, since $|U|  =  \binom{s + t}{t}$, we can conclude that $U$ contains a blue $K_t$ on vertices $u_1,\ldots, u_t \in U$.  For $u_i \in U$, let

$$N_r(u_i) = \{v \in V: \chi(u_i,v) = \textnormal{red}\}.$$

\noindent Then we have $|N_r(u_i)| < r_2(K_{s-1},P_n^t).$  Let

$$V_1' = V_1 \setminus(N_r(u_1)\cup\cdots \cup N_r(u_t)),$$

$$V_2' = V_2 \setminus(N_r(u_1)\cup\cdots \cup N_r(u_t)).$$

Then notice that we must have either $|V_1'| < r_2(K_s,P_{\lfloor n/2\rfloor}^t)$ or  $|V_2'| < r(K_s,P_{\lfloor n/2\rfloor}^t)$.  Indeed, otherwise both $V_1'$ and $V_2'$ contain a blue $P_{\lfloor n/2\rfloor}^t$. Since $\chi$ colors all edges between $u_i$ and $V'_1\cup V_2'$ blue, we can combine both blue copies of $P_{\lfloor n/2\rfloor}^t$ with vertices $u_1,\ldots, u_t$ and obtain a blue $P_{2\lfloor n/2\rfloor +t}$, which contains a copy of a blue $P_n^t$ since $2\lfloor n/2\rfloor + t > n$.    

Therefore, without loss of generality, we can assume that $|V'_1| < r_2(K_s,P_{\lfloor n/2\rfloor}^t)$.  On the other hand, we have

$$|V'_1| \geq \lfloor N/2\rfloor  - \binom{s + t}{t} - t\cdot r_2(K_{s-1},P_n^t).$$

\noindent Hence

$$ N \leq 2 r_2(K_s,P_{\lfloor n/2\rfloor}^t) + 2\binom{s + t}{t} + 2t\cdot r_2(K_{s-1},P_n^t).$$

\noindent By the induction hypothesis, we have

$$N \leq t^{4s}n(\log n - 1)^{s-2} + 2\cdot 4^s + 2t\cdot t^{4s - 4} n(\log n)^{s-3}.$$

$$ \leq t^{4s}n(\log n)^{s-2} - (s-2)t^{4s}n(\log n)^{s-3} + (s-2)^2t^{4s}n(\log n)^{s-4} + 2\cdot 4^s + 2t^{4s  - 3}n (\log n)^{s-3}$$

$$\leq t^{4s}n(\log n)^{s-2}.$$

\end{proof}

The proof of Theorem \ref{main2} is very similar to the argument above.

\begin{proof}[Proof of Theorem \ref{main2}]  We proceed by induction on $n$.  The base case $n = 2$ is trivial.  Now assume that the statement holds for all $n' < n$.  Set $N = (2s)^{t(t + 1)\log n}$.  We start with a standard supersaturation argument.  For sake of contradiction, suppose there is a red/blue coloring $\chi:\binom{[N]}{2}\rightarrow \{\textnormal{red,blue}\}$ of the pairs of $[N]$ such that $\chi$ does not produce a red $K_s$ nor a blue $P_n^t$.  Let $r = r(K_s,K_{t + 1})$.  Then we must have at least

$$\frac{\binom{{N}}{r}}{\binom{N- (t + 1)}{r -(t + 1)}} = \frac{N!}{r!}\frac{(r - (t + 1))!}{(N - (t + 1))!} \geq \frac{(N - t)^{t + 1}}{r^{t +1}} \geq\frac{N^{t + 1}}{(2r)^{t  + 1}}$$

\noindent copies of $K_{t + 1}$.  For each blue copy of $K_{t + 1}$ with vertex set $x_0 < x_1 < \cdots < x_{t}$, we associate the middle $t-1$ vertices $\{x_1,\ldots, x_{t -1}$.    By the pigeonhole principle, there is a set $Y = \{x_1,x_2,\ldots, x_{t - 1}\}$ with $x_1 < x_2 < \cdots < x_{t -1}$, such that $Y$ is the middle set for at least

$$\frac{N^{t + 1}}{(2r)^{t + 1}}\frac{1}{N^{t-1}} \geq \frac{N^2}{(2r)^{t + 1}}$$

\noindent blue copies of $K_{t + 1}$.  Let $V_1 \subset \{1,2,\ldots,x_1 - 1\}$ such that $x \in V_1$ if there is a blue $K_{t + 1}$ whose middle set is $Y$ and $x$ is the first vertex of the blue $K_{t + 1}$.  Likewise, let $V_2 \subset  \{x_{t - 1} + 1,\ldots, N\}$ such that $x \in V_1$ if there is a blue $K_{t + 1}$ whose middle set is $Y$ and $x$ is the last vertex of the blue $K_{t + 1}$.  Hence, we have

$$|V_1||V_2| \geq \frac{N^2}{(2r)^{t + 1}}.$$

\noindent Moreover, $\chi$ colors all edges between $V_1$ and $Y$ blue, and all edges between $V_2$ and $Y$ blue.  Since $|V_1|,|V_2| < N$, we must have $|V_1|,|V_2| \geq \frac{N}{(2r)^{t + 1}}$.  Since the Erd\H os-Szekeres theorem implies that $r_2(K_s,K_{t + 1}) \leq \binom{s + t - 1}{t} \leq s^t$, we have

$$\min \{|V_1|,|V_2|\} \geq \frac{N}{(2s)^{t (t + 1)}} = \frac{(2s)^{t(t + 1)\log n}}{(2s)^{t(t + 1)}} \geq (2s)^{t(t + 1)\log \lfloor n/2\rfloor }.$$

\noindent By the inductive hypothesis, both $V_1$ and $V_2$ contain a blue $P_{\lfloor n/2\rfloor}^t$.  Together with the vertices in $Y$, we obtain a blue copy of $P_{2\lfloor n/2\rfloor + t-1}^t$.  Since $2\lfloor n/2\rfloor + t-1 \geq n$, this completes the proof.
\end{proof}


\end{document}